\documentclass[12pt]{amsart}
\usepackage{amssymb,mathrsfs,color}
\usepackage{a4wide, bbm, calligra, }
\usepackage[2cell,arrow,curve,matrix]{xy}
\UseHalfTwocells \UseTwocells

\newtheorem{Pro}{Proposition}[subsection]
\newtheorem{Le}[Pro]{Lemma}
\newtheorem{Th}[Pro]{Theorem}
\newtheorem{Co}[Pro]{Corollary}
\theoremstyle{definition}

\theoremstyle{remark}
\newtheorem{Exm}[Pro]{Example}

\def\af{\frak a}
\def\fb{\frak b}

\def\fa{\frak F}

\def\ea{{\mathscr E}}

\def\fa{{\mathscr F}}

\def\Loc{\mathfrak{Loc}}

\def\fork{\pitchfork}

\let\t\otimes

\let\xto\xrightarrow

\def\0{^{[1]}}

\def\1{^{-1}}

\def\cref#1#2#3{\left(#2\right.\left|\ #3\right)_{#1}}

\def\hi{\mathfrak{I}}

\def\ii{\mathscr{I}\hspace{-0.2em}\mathscr{I}}

\def\ta{{\mathscr T}}

\def\ope{{\bf Open}}

\def\N{{\mathbb N}}

\def\Sh{{\bf Sh}}

\def\F{{\mathsf F}}

\def\Spec{{\rm Spec}}

\def\Sets{{\mathscr{Sets}}}

\def\set{{\sf {Sets}}}

\def\Pts{{\bf Pts }}

\def\Hom{{\sf Hom}}

\def\sF{\mathsf{F}}

\def\hom{{\sf Hom}}

\numberwithin{equation}{subsection}

\def\fa{\mathfrak{a}}
\def\fm{\mathfrak{m}}
\def\cF{\mathcal{F}}

\begin{document}

\title{Idempotents and the points of the topos of M-sets}

\author[I. Pirashvili]{Ilia Pirashvili}

\maketitle

\begin{abstract} The aim of this paper is to study the points and localising subcategories of the topos of $M$-sets, for a finite monoid $M$. We show that the points of this topos can be fully classified using the idempotents of $M$. We introduce a topology on the iso-classes of these points, which differs from the classical topology introduced in SGA4. Likewise, the localised subcategories of the topos $M$-sets correspond to the set of all two-sided idempotent Ideals of $M$.
\end{abstract}

\section{Introduction}

For a topos $\ea$, we denote the category of points of $\ea$ by $\Pts(\ea)$ and the isomorphism classes of $\Pts(\ea)$ is denoted by $\F(\ea)$. In SGA4 \cite{SGA4}, the authors defined a topological space structure on the set $\F(\ea)$, based on the subobjects of the terminal object of $\ea$. In the case when $\ea$ is the topos ${\bf Sh}(X)$ of sheaves on a sober topological space $X$, the space constructed in \cite{SGA4} is homeomorphic to $X$. However, this topology is quite trivial in many cases. Particularly, when $\ea=\Sets_M$ is the topos of right $M$-sets. In this case, the obtained topological space has only two open sets, though $\F(M)$ could be quite big. Hence, it is natural to search for other topologies on $\F(\ea)$.

In recent years, there was a considerable attention to the problem of understanding the points and possible topological space structures on $\ea=\Sets_M$. When $M=\N_+^{\times }$ is the multiplicative monoid of strictly positive natural numbers there is an interesting topological space structure on the set $\F(\Sets_M)$ studied by  Connes-Consani \cite{arith_site} and  Le Bruyn \cite{llb}, see also a related paper by Hemelaer \cite{h}.

On the other hand, there exists the very interesting topological space $\Spec(M)$, for a commutative monoid $M$. The elements of  $\Spec(M)$ are prime ideals of $M$ \cite{kato} and the topology is defined exactly as in classical algebraic geometry. This space plays an important role in $F_1$-mathematics and $K$-theory, see for example \cite{deitmar},\cite{f1zeta},\cite{corinas}.

In our previous papers \cite{points}, \cite{reconstruction}, we reconstructed the topological space $\Spec(M)$, from the topos $\Sets_M$, when $M$ is finitely generated and commutative. In fact, we constructed a bijections of sets $\Spec(M)\to \F(\Sets_M)$ and 
${\sf Open}(\Spec(M)) \to {\Loc}(\Sets_M)$, where ${\Loc}(\ea)$ is the set of localised subcategories of a topos $\ta$ and ${\sf Open}(X)$ is the set of all open subsets of a topological space $X$.

We introduce the following notation: For a localising subcategory $\ta$ of $\ea$ and $p=(p_*,p^*):\set \to\ea$ a topos point of $\ea$, we write $p\fork \ta$ if $p_*(S)\in \ta$ for every set $S\in\set$.

The aim of this work is to investigate the points and localising categories of $\Sets_M$, when $M$ is a finite monoid. We show that localising subcategories induces an interesting topology on the set $\F(\Sets_M)$. In more details, we construct a bijection from the set $\F(\Sets_M)$ to the set of ${\mathfrak{I}}$-classes of idempotents, where ${\mathfrak{I}}$ is the Green relation (that is, $e{\mathfrak{I}} f$ if and only if $MeM=MfM$). Since this set has a natural order $\leq_{\mathfrak{I}}$ ($e\leq_{\mathfrak{I}} f$ if $MeM\subseteq MfM$), one can consider the order topology on $\F(\Sets_M)$. Our next result claims that there are bijections from the set of localised subcategories ${\Loc}(\Sets_M)$ to the set of all two-sided idempotent ideals of $M$, and also to the set of all open subsets of the ordered topology on $\F(\Sets_M)$.

\section{Preliminaries}

\subsection{Points and filtered $M$-sets}

Recall that a \emph{point} of a Grothendieck topos, or simply a topos $\ta$, is a geometric morphism ${\sf p}=({\sf p}_*,{\sf p}^*):\Sets\to \ta$ from the topos of sets $\Sets$ to $\ta$. The inverse image functor ${\sf p}^*:\ta\to\Sets$ preserves colimits and finite limits. It is also well-known that, conversely, for any functor ${\sf f}:\ta\to\Sets$ which preserves colimits and finite limits, one has ${\sf f}={\sf p}^*$ for a uniquely defined point ${\sf p}$. We let $\Pts(\ta)$ be the category of points of $\ta$ and $\sF_\ta$ the isomorphism classes of the category $\Pts(\ta)$.

Let $M$ be a monoid. The category of left (resp. right) $M$-sets is denoted by $_M\Sets$ (resp. $\Sets_M$). It is well known that these categories are topoi.   Instead of $\Pts(\Sets_M)$  and $\sF_{\Sets_M}$ we write $\Pts(M)$ and $\sF_M$. By Diaconescu's theorem  \cite{mm}, the category $\Pts(M)$ is equivalent to the category of filtered left $M$-sets. Recall that a left $M$-set $A$ is called \emph{filtered}, provided the functor
$$(-)\otimes_M A:\, _M\Sets\to \Sets$$
commutes with finite limits. The  topos point of $\Sets_M$ corresponding to a filtered left $M$-set $A$ is denoted by ${\sf p}_A=({\sf p}_{A*},{\sf p}_A^*)$. The inverse image functor ${\sf p}^*_A: \Sets_M\to \Sets $ is given by
$${\sf p}^*_A(X)=X\otimes_MA,$$
while the direct image functor ${\sf p}_{A*}:\Sets\to \Sets_M$ sends a set $Y$ to $\Hom_\Sets(A,Y)$. The latter is considered as a right $M$ set via
$$(\alpha m)(a):=\alpha(ma).$$
Here, $\alpha\in \Hom_\Sets(A,X)$, $a\in A$ and $m\in M$.
 
The following, well-known, fact \cite[p.24]{m} is a very useful tool for checking whether a given $M$-set is filtered.
 
\begin{Le}\label{Fcond} A left $M$-set $A$ is filtered if and only if the following three conditions hold:
\begin{itemize}
	\item[(F1)] $A\not =\emptyset$.
	\item[(F2)] If $m_1,m_2\in M$ and $a\in A$ satisfies the condition
	$$m_1a=m_2a,$$
	there exist $m\in M$ and $\tilde{a}\in A$, such that $m\tilde{a}=a$ and $m_1m=m_2m$.
 	\item[(F3)] If $a_1,a_2\in A$, there are $m_1,m_2\in M$ and $a\in A$, such that $m_1a=a_1$ and $m_2a=a_2$.
\end{itemize}
\end{Le}

\begin{Exm}\label{222} i) Clearly, $A=M$ is always filtered and corresponds to the canonical point, denoted by ${\sf p}_M$. Thus, ${\sf p}_M^*$ is the forgetful functor $\Sets_M\to \Sets$.
	
ii) We take $M=\{1,t\}$, $t^2=t$. In this case, the singleton, which is a terminal object in $\Sets_M$, is a filtered $M$-set. 
\end{Exm}

\subsection{Points and prime ideals}

If $M$ is commutative and $\mathfrak p$ is a prime ideal, then the localisation $M_{\mathfrak p}$ is filtered \cite{points} and in this way, one obtains an injective  map
$$Spec(M)\to \F_M.$$
The inverse image functor corresponding to the point, associated to the filtered $M$-set $M_{\mathfrak p}$, sends a right $M$-set $X$ to the localisation $X_{\mathfrak p}$, considered as a set.

Moreover, if $M$ is finitely generated, the map is a bijection.
\\

For a monoid $M$, denote by $M^{com}$ the maximal commutative and by $M^{sl}$ the maximal semilattice quotient respectively. As a semilattice is commutative by definition, we have natural surjective homomorphisms $M\to M^{com}\to M^{sl}$.

According to \cite{p1}, for any commutative monoid $M$, the induced map
$$Spec(M^{sl})\xto{\cong} Spec(M)$$
is bijective and furthermore, there is an injective map \cite{p1}
$$M^{sl}\to Spec(M^{sl})\cong Spec(M).$$
This is bijective if $M$ is commutative and finitely generated. It follows that under these assumptions, we have
$$|\F_M|=|M^{sl}|.$$

\begin{Exm}\label{231} Let $M=\{1,t\}$, $t^2=t$ as in part ii) of Example \ref{222}. Then $\F_M$ has only two elements, one corresponds to the filtered $M$-set $M$ and another one to the singleton.
\end{Exm}

\subsection{Induced points}

We will discuss some functorial properties of $\F_M$ and its consequences. We start with the following well-known fact:

\begin{Le} Let $f:M\to M'$ be a monoid homomorphism. For any filtered left $M$-set $A$, the left $M'$-set $M'\otimes_M A$ is filtered.
\end{Le}

The $M'$-set constructed in the Lemma is said to be \emph{induced from $A$ via the homomorphism $f$}. In this way, one obtains a functor
$$\Pts(f): {\Pts}(M) \to {\Pts}({M'})$$
and the induced map
$$\F_f:\F_M\to \F_{M'}.$$

\begin{Exm} Let $e$ be an idempotent of a monoid $M$. We have a homomorphism of monoids $\eta:\{1,t\}\to M$, where $t^2=t$ and $\eta(t)=e$. The singleton is filtered over $\{1,t\}$, see Example \ref{231}. Thus, it induces a filtered $M$-set, which is easily seen to be $Me$. The fact that an idempotent of $M$ gives rise to a point of $\Sets_M$ was already recently observed in  \cite{rogers},\cite{HR}.

One of our main result claims that if $M$ is finite, any filtered $M$ set is isomorphic to $Me$ for an idempotent $e\in M$; that is, induced from the submonoid $\{1,e\}$, see Theorem \ref{7.102010} below.
\end{Exm}

Denote by $q:M\to M^{sl}$ the quotient map. As we said, the induced map
$$\F_q :\F_M\to  \F_{M^{sl}}$$
is bijective if $M$ is finitely generated and commutative. In the noncommutative setting, this induced map is not a bijection in general, even if $M$ is finite. However, in this case, we will show that $|\F_M|$ is finite, see Theorem \ref{132'30.9} below. We also have the following:

\begin{Pro}\label{142} Let $M$ be a finite monoid. The canonical homomorphism
$$\F_q:\, \F_M\to  \F_{M^{sl}}$$
is surjective.
\end{Pro}

\begin{proof} Denote by $\mathsf{Idem}(M)$ the set of all idempotents of $M$. Then $\mathsf{Idem}(M)\ni e\mapsto Me$ yields the map	$\mathsf{Idem}(M)\to \F_M$. Since $\mathsf{Idem}(M^{sl})=M^{sl}$, the functoriality yields the following commutative diagram
$$\xymatrix{ \mathsf{Idem}(M)\ar[r]\ar[d] & \F_M\ar[d]\\
		     M^{sl}\ar[r]& \F_{M^{sl}}.}$$
The bottom arrow is a bijection and the left vertical map is surjective, thanks to \cite[Lemma 1.6] {st}. It follows that the left vertical map is surjective as well.
\end{proof}

\section{Points of $\Sets_M$ and idempotents of $M$}

\subsection{Category $\hi(M)$}

Let $m\in M$ be an element. We have a natural left ideal $Mm$, which can also be considered as a left $M$-set. We have the following well-known fact:

\begin{Le} \label{eX} Let $e\in M$ be an idempotent. For any left $M$-set $X$, we have a bijection
$$eX\cong Hom_M(Me,X),$$
which sends an element $ex\in eX$ to the homomorphism $\alpha_{ex}:Me\to X$, given by $\alpha_{ex}(me)=mex$.  
\end{Le}

\begin{proof} Take any morphism of $M$-sets $\beta:Me\to X$. We have $\beta(e)=\beta(ee)=e\beta(e)$. Thus, $\beta(e)\in eX$ and $\beta\mapsto \beta(e)$ defines a map $Hom_M(Me,X)\to eX$, which obviously is inverse of the map $ex\mapsto \alpha_{ex}$.
\end{proof}

\begin{Co}\label{eMf} Let $e,f\in M$ be idempotents. We have $Hom_M(Me,Mf)=eMf.$
\end{Co}

We define the category $\hi(M)$ as follows: Objects are idempotents of $M$ and a morphism from $e$ to $f$ is an element of $M$ of the form $fme$, $m\in M$. That is
$$\hom_{\hi(M)}(e,f)=fMe.$$
The composition is given by the multiplication in $M$, i.e. $(gnf)\circ (fme)=gnfme$. That is, the composite of arrows
$$e\xto{fme}f\xto{gnf}g$$
is equal to $e\xto{gnfme}g$.

The identity morphism of $e$ is just $e$. It is clear that
$$\hi(M^{op})=(\hi(M))^{op}.$$

Recall the Green relations $\mathfrak{L}, \mathfrak{R}, \mathfrak{I}.$  By definition, we have $e\mathfrak{L} e'$ provided $Me=Me'$,  $e\mathfrak{R} e'$ if $eM=e'M$ and $e\mathfrak{I} e'$ if $MeM=Me'M$. We let 
$\mathsf{Idem}_\mathfrak{K}(M)$ be the corresponding quotient set, where $\mathfrak{K}\in \{\mathfrak{L}, \mathfrak{R}, \mathfrak{I}\}.$ 

It is clear that if $e\mathfrak{L} e'$ and $f\mathfrak{R} f'$, then
$$\hom_{\hi(M)}(e,f)=\hom_{\hi(M)}(e',f) \quad  \hom_{\hi(M)}(e,f)=\hom_{\hi(M)}(e,f').$$

\begin{Le} The assignment $e\mapsto Me$ induces a contravariant duality between the category $\hi(M)$, and, the full subcategory of $_M\Sets$, consisting of objects of the form $Me$.
\end{Le}

\begin{proof} This is a direct consequence of Corollary \ref{eMf}.
\end{proof}

\begin{Le} Idempotents $e$ and $f$ are isomorphic in $\hi(M)$ if and only if there are $a,b\in M$, such that $ab=e$ and $ba=f$. This happens if and only if $e\mathfrak{I} f$. Thus, the set of iso-classes of the category $\hi(M)$ is bijective to the set $\mathsf{Idem}_\mathfrak{I}(M)$.
\end{Le}

\begin{proof} A consequence of \cite[Theorem 1.11]{st}.
\end{proof}

\subsection{$M$-congruences}

For an equivalence relation $\rho$ on a set $S$, we denote by $q$ the canonical map $q:S\to S/\rho$.

Let $M$ be a monoid and $A$ a left $M$-set. An equivalence relation $\sim_\rho$ is called an $M$-\emph{congruence} if $a\sim b$ implies $ma\sim mb$ for all $m\in M$. It is clear that in this case, the quotient $A/\rho$ has an unique left $M$-set structure such that $q:M\to M/\rho$ is an $M$-set map.

We will use this terminology, to distinguish between congruences on a monoid $M$, in the world of monoids, and congruences on $M$, considered as a left $M$-set, using the multiplication in $M$.

\begin{Le}\label{111.30.9} Let $\sim_\rho$ be an $M$-congruence on a left $M$-set $A$. For any $a\in A$, the subset
$$K_\rho(a)=\{x\in M| a \sim xa \}$$
is a submonoid of $M$.
\end{Le}

\begin{proof} Since $\sim$ is an equivalence relation, we have $a\sim a=1\cdot a$. Thus, $1\in K$. Assume $x,y\in K(a)$. That is, $a\sim ya$ and $a\sim xa$. Since $\sim$ is $M$-congruence, we have $xa\sim xya$. It follows that $a\sim xa\sim xya$. Hence, $xy\in K(a)$. 
\end{proof}

\subsection{Finitines of points}

We start with the following observation:

\begin{Le}\label{sourceish} Let $A$ be a filtered left $M$-set and $a_1,\cdots, a_k$ elements of $A$. There are $m_1,\cdots, m_k\in M$ and $c\in A$, such that $a_i=m_ic$ for all $1\leq i\leq k$.
\end{Le}

\begin{proof} The case $n=1$ is clear and $n=2$ is just condition F3. We proceed by induction. Assume there are $n_1,\cdots, n_{k-1}\in M$ and $b\in A$, such that $a_i=n_ib$ for all $1\leq i\leq n-1$. By the case $n=2$, we can choose $c\in A$ and $n',m_k\in M$, such that $n'c=b$ and $m_kc=a_k$. We put $m_i=n'n_i$ for $i=1, \cdots k-1$ to obtain $m_ic=a_i$ for all $1\leq i\leq k$.
\end{proof}

\begin{Th}\label{132'30.9} Let $M$ be a finite monoid.
	\begin{itemize}
	\item [i)] Any filtered $M$-set $A$ is cyclic, that is, generated by a single element. In particular, we have $|A|\leq |M|$.
	\item [ii)] The set $\F_M$ is finite.
	\end{itemize}
\end{Th}

\begin{proof} i) Assume there are $k$ distinguished elements of $A$, say $a_1,\cdots, a_{k}$. By Lemma \ref{sourceish}, we can find $m_1,\ldots, m_k\in M$ and $c\in A$, such that $a_i=m_ic$, for all $1\leq i\leq k$. If $|M|<k$, we see that there are $i\not =j$, such that $m_i=m_j$. It follows that
$$a_i=m_ic=m_jc=a_j.$$
This contradicts our assumptions on the $a_i$'s. Hence, $k\leq |M|$. This implies $|A|\leq |M|$. Moreover, by taking $a_1,\cdots, a_{k}$ to be all the elements of $A$, we see that $A$ is generated by $c$ and hence, it is cyclic.

ii) This is an obvious consequence of i).
\end{proof}

\subsection{$F$-monoids and $F$-submonoids}

We call a monoid $K$ an \emph{$F$-monoid}, if for any $m,n\in K$, there exits an $x\in K$, such that 
$$mx=nx.$$
Clearly, $K=\{1,e\}$ is an $F$-monoid, where $e^2=e$. More generally, if $M$ is a semilattice, then $M$ is an $F$-monoid (we can take $x=mn$).

Another class of $F$-monoids (which is a generalisation of semilattices in the finite case) are monoids that have a right zero. That is, an element $\varrho$, such that $x\varrho=\varrho$ for all $x\in K$. Clearly, if such a $\varrho$ exists, it is unique and an idempotent. Our next goal is to show that the converse is also true, if $M$ is finite. We will need the following lemma.

\begin{Le}\label{141} Let $K$ be an $F$-monoid. For any finite collection of elements $m_1,\cdots, m_k$ of $K$, there exists an $x\in K$, such that $m_ix =x$ for all $i=1,\cdots,k$.
\end{Le}

\begin{proof} We proceed by induction. Let $k=1$. In this case, the assertion follows directly from the definition of an $F$-monoid, by taking $m=m_1$ and $n=1$. Next, consider the case $k=2$. Since the assertion is true for $k=1$, we can find $y_i,$ $i=1,2$, such that $m_iy_i=y_i$, $i=1,2$. As $K$ is an $F$-monoid, there exits a $z\in K$, such that $y_1z=y_2z$. Now, for $x=y_1z=y_2z$, we have
$$m_ix=m_iy_iz=y_iz=x, \quad i=1,2.$$
Let $k>2$. By the induction assumption, there exists a $y\in K$, for which $m_iy=y$, for all $i=1,\cdots, k-1$. Since the result is also true for $k=2$, we can apply it for $y$ and $m_k$, to conclude that there exists a $x\in K$, for which $m_kx=x$ and $yx=x$. We have $m_ix=m_iyx=yx=x$, $i<k$ and thus, for all $1\leq i\leq k$. This finishes the proof.
\end{proof}

\begin{Co}\label{142.7.10} A finite monoid is an $F$-monoid if and only if it has a right zero element.
\end{Co}

A submonoid $K$ of a monoid $M$ is called an \emph{$F$-submonoid} if $K$ is an $F$-monoid.

\subsection{Saturated submonoids}
 
A submonoid $K$ of a monoid $M$ is called \emph{saturated}, if for any $m\in M$, for which $mx=x$ for some $x\in K$, one has $m\in K$. It is clear, that to any submonoid $K$, there is a smallest saturated submonoid containing $K$. This is the intersection of all saturated submonoids containing $K$. (The fact that the intersection of saturated submonoids is again saturated is readily checked). We denote this associated saturated submonoid by $\widehat{K}$ and sometimes refer to it as the \emph{saturation} of $K$.

\begin{Exm} Let $M=\{1,0,a,b,ab\}$, where $a^2=a, b^2=b$ and $ba=0.$ The only saturated $F$-submonoids of $M$ are $\{1\}, \{1,a\}, \{1,b\}$ and $M$ itself. On the other hand, $\{1,0\}, \{1,0,a\}, \{1,0,b\}$ and $\{1,0,ab\}$ are non-saturated $F$-submonoids.
\end{Exm}

\subsection{Quotient by $F$-submonoids}

Let $H\subseteq A$ be a submonoid of a monoid $A$. We can define a relation $\sim$ on $A$ by $a\sim_H b$ if and only if there exist $x,y\in H$, such that $ax=by$. This relation, however, need not be an equivalence relation, as the transitivity property need not hold. If it does, however, than it is an $M$-congruence and the quotient, denoted by $M/K$, is a natural $M$-set.

\begin{Le}\label{filteredismonoid} Let $K$ be an $F$-submonoid of $M$ and $m,n\in M$. Then its induced relation $\sim_K$ is an $M$-congruence and $m\sim_K n$ if and only if there exists an $x\in K$, such that $mx=nx$.
\end{Le}

\begin{proof} First, we show that the two relations are the same. One side is clear, by letting $x=y$. For the other side, let $x,y\in K$ be elements, such that $mx=ny$. Since $K$ is an $F$-submonoid, we can find a $z\in K$, such that $xz=yz$. As $mx=ny\Rightarrow mxz=nyz$, the result follows.

To see that $\sim_K$ is an $M$-congruence, let $m\sim_K n$ and $n\sim_K k$. That is, we have $x,y\in K$, such that $mx=mx$ and $ny=ky$. As $K$ is an $F$-submonoid, there exists a $z\in K$, such that $xz=yz$. We now have
$$m(xz)=(mx)z=(nx)z=n(xz)=n(yz)=(ny)z=(ky)z=k(yz).$$
Hence, $m\sim_K k$ with the first definition and so, transitivity holds.
\end{proof}

In general, the $M$-congruence $\sim_K$ is not a congruence (that is, the quotient $M/K$ need not be a monoid), even if $K$ is an $F$-submonoid. In fact, take again $M=\{1,0,a,b,ab\}$, where $a^2=a, b^2=b$ and $ba=0.$ We take $K=\{1,b\}$. In this case, we have three equivalence classes $1\sim_K b$, $a\sim_K ab$ and $0$. Since $1\cdot a=a \not \sim_K ba=0$, we see that $\sim_K$ is not a congruence.

We have shown that $M/K$ is a left $M$-set. More is true, however.

\begin{Pro}\label{153} Let $K$ be an $F$-submonoid of $M$. The quotient $M/K$ is a filtered left $M$-set.
\end{Pro}

\begin{proof} The fact that $M/K$ is a left $M$-set is just Lemma \ref{filteredismonoid}. To see that $M/K$ is in addition filtered, consider the canonical surjective map $q:M\to M/K$. Since $M$ is a monoid, it is not empty. Hence, $M/K$ is also non-empty and condition F1 holds.

To show F2, assume $m_1q(a)=m_2q(a)$, with $a,m_1,m_2\in M$. We have to find $m, \tilde{a}\in M$, such that $m_1m=m_2m$ and $mq(\tilde{a})=q(a)$. To this end, observe that
$$q(m_1a)=m_1q(a)=m_2q(a)=q(m_2a).$$
Hence, $m_1ax=m_2ax$ for an element $x\in K$. Since $K$ is an $F$-submonoid, there exists an element $y\in K$, such that $xy=y$. So, we can take $m=ax$ and $\tilde{a}=q(1)$. We have $m_1m=m_1ax=m_2ax=m_2m$ and $mq(1)=q(ax)=q(a)$, because $axy=ay$ and $ax\sim_K a$. 

To show F3, take $q(a_1),q(a_2)\in M/K$. Then $a_1q(1)=q(a_1)$ and $a_2q(1)=q(a_2)$. Thus, F3 holds with $a=q(1)$. 
\end{proof}

We have the following ``inverse statement'' to the above:

\begin{Pro}\label{154} Let $\rho$ be an $M$-congruence on $M$, such that $M/\sim_\rho$ is a filtered left $M$-set and define
$$K_\rho=\{m\in M|1\sim_\rho m\}.$$
\begin{itemize}
	\item[i)] The subset $K_\rho$ is a saturated $F$-submonoid of $M$.
	\item[ii)] We have $M/\sim_\rho \ \simeq M/K_\rho$. 
\end{itemize}

\end{Pro}

\begin{proof} i) By Lemma \ref{111.30.9}, $K_\rho$ is a submonoid of $M$. Take $x,y\in K_\rho$. We have $x\sim_\rho 1\sim_\rho y$, by the definition of $K_\rho$. Thus $q(x)=q(y)$, where $q:M\to M/\sim_\rho$ is the canonical map. It follows that $xq(1)=yq(1)$. Since $M/\sim_\rho$ is filtered, there are $m, \tilde{a}\in M$ such that $xm=ym$ and $mq(\tilde{a})=q(1)$. The last condition implies that $z:=m\tilde{a}\in K$. We have $xz=xm\tilde{a}=ym\tilde{a}=yz$. This shows that $K_\rho$ is an $F$-submonoid.

Take $m\in M$ and $x\in K_\rho$, such that $mx=x$. Since $1\sim_\rho x$, it follows that $m\sim_\rho mx=x\sim_\rho 1$. Thus, $m\in K_\rho$ and $K_\rho$ is saturated.

ii) Assume $m\sim_\rho n$. It follows that $mq(1)=nq(1)$. By F2, there are $x,a\in M$, such that $mx=nx$ and $xq(a)=q(1)$. The last condition implies $y=xa\in K$. Since $my=mxa=nxa=ny$, we see that $m\sim_K n$. Conversely, let $m\sim_K n$. That is $mx=nx$ for $x\in K$. We have $q(mx)=mq(x)=mq(1)=q(m)$ and $q(nx)=q(n)$. This yields $q(m)=q(n)$, which implies $m\sim_\rho n$. This finishes the proof.
\end{proof}

\begin{Co} Let $K$ be an $F$-submonoid of $M$ and $\rho$ be the congruence $\sim_K$, corresponding to $K$. Then  
\begin{itemize}
	\item[i)] $K_\rho=\{m\in M| mx=x, \ {\rm for \ an \ element} \ x\in K\}=\widehat{K}, $ where $\widehat{K}$ is the saturation of $K$,
	\item[ii)] $\widehat{K}$ is an $F$-submonoid,
	\item[iii)] if $K$ is saturated, we have $\widehat{K}=K$.
\end{itemize}	
\end{Co}

\begin{proof} i) By definition, $m\in K_\rho$ if and only if $1\sim_\rho m$. This happens exactly if there exists an $x\in K$, such that $mx=x$. This proves the first equality. Next, take $m\in K$. Since $K$ is an $F$-monoid, there exists an $x\in K$, such that $mx=x$. Hence, $1\sim_\rho m$ and $m\in K_\rho$. This shows that $K\subseteq K_\rho$. Since $K_\rho$ is saturated by Proposition \ref{154}, we have $\widehat{K}\subseteq K_\rho$. To show the opposite inclusion, take any $m\in K_\rho$. Thus, $1\sim_\rho m$. So, $m=mx$ for some $x\in K$. Since $\widehat{K}$ is saturated and $x\in K\subseteq \widehat{K}$, we see that $m\in \widehat{K}$. It follows that $K_\rho=\widehat{K}$.

ii) Follows from part i) by virtue Proposition \ref{154} and iii) is holds by definition.
\end{proof}

\begin{Th}\label{141.0110} Let $M$ be a finite monoid. Any filtered left $M$-set is isomorphic to $M/K$, for a suitable saturated $F$-submonoid $K$ of $M$.
\end{Th}

Below (see Theorem \ref{7.102010}), we will prove a much stronger result.

\begin{proof}Any filtered $M$-set is cyclic by Theorem \ref{132'30.9}. Hence, it is isomorphic to $M/\rho$, where $\rho$ is an $M$-congruence. Part ii) of Proposition \ref{154} says that $K=K_\rho$ is a saturated $F$-submodule of $M$. It remains to show that $M/\sim_\rho\simeq M/K_\rho$, which we already did in Proposition \ref{154}.
\end{proof}

\begin{Co}\label{153710} Let $K$ and $L$ be $F$-submonoids of $M$. Then $M/K\simeq M/L$ if and only if $\widehat{K}=\hat{L}$. In particular, $M/K\simeq M/\widehat{K}$.
\end{Co}

\subsection{ Some examples}

i) Assume $M$ is an $F$-monoid and take $K=M$. In this case, $M/M$ is a singleton. So, the terminal object of $\Sets_M$ is filtered. Conversely, if the terminal object of $ \Sets_M$ is filtered, then $M$ is an $F$-monoid. In fact, the terminal object is $M/\sim_\rho$, where $x\sim_\rho y$ for all $x,y\in M$. In this case, $K=M$ and hence, $M$ is an $F$-monoid, thanks to the proof of Theorem \ref{141.0110}. In particular, the single element set is filtered if $M$ has a right zero or $M$ is a semilattice.

ii) Let $e\in M$ be an idempotent. Then $K=\{1,e\}$ is an $F$-submonoid. In this case, $\sim_K$ is the equivalence relation on $M$, defined by $a\sim_{K} b$ if and only if $ae=be$. The left $M$-set $M/K\simeq Me$ is filtered. 

\begin{Th}\label{7.102010} Let $M$ be a finite monoid. Any filtered $M$-set is of the form $Me$, for an idempotent $e\in M$. Thus, $e\mapsto Me$ yields an equivalence of categories
$$ \left(\hi(M)\right)^{op}\cong\Pts(M).$$
\end{Th}

\begin{proof} We have already proven (see Theorem \ref{141.0110}) that any filtered $M$-set is of the form $M/K$, where $K$ is a saturated $F$-submonoind of $M$. By Corollary \ref{142.7.10}, $K$ has a right zero $\varrho$. Take $L=\{1,\varrho\}$. We have $L\subseteq K$ and $K\subseteq \widehat{L}$. Since $K$ is saturated, we have $K=\widehat{L}$. By Corollary \ref{153710}, we have $\sim_K=\sim_L$. Thus, we can take $e=\varrho$. 
\end{proof}

\begin{Co}\label{372.08} Let $M$ be a finite monoid, then 
\begin{itemize}
	\item[i)] $|\Pts(M)|=|\Pts(M^{op})|.$
	\item[ii)] If ${\sf p}$ is a point corresponding to an idempotent $e$, one has an isomorphism of monoids 
	$$End({\sf p})\cong \left ( eMe\right )^{op}.$$
	\item[iii)] There is a bijection $$\F_M\cong \mathsf{Idem}_{\mathfrak{I}}(M).$$
\end{itemize}
\end{Co}

\section{Lattice of localising subcategories in $\Sets_M$}

\subsection{General facts on localising subcategories}

Recall that a \emph{localising subcategory} $\ta$ of a Grothendieck topos $\ea$ is a full subcategory of $\ea$, such that the following conditions hold:

\begin{itemize}
	\item [ i)]  If $x$ belongs to $\ta$ and $y\in \ea$ is isomorphic to $x$, then $y$ belongs to $\ta$.
	\item [ ii)] The inclusion $\iota:\ta\to \ea$ has a left adjoint $\rho:\ea\to \ta$, called the \emph{localisation}.
	\item [ iii)] The localisation $\rho$ respects finite limits.
\end{itemize}

It is well-known that in this case, $\ta$ is also a Grothendieck topos. We denote by $\Loc(\ea)$ the poset of all localising subcategories of $\ea$.

For a localising subcategory $\ta$ of $\ea$ and $p=(p_*,p^*):\set \to\ea$ a topos point of $\ea$, we write $p\fork \ta$ if $p_*(S)\in \ta$ for every set $S\in\set$.

We now consider the case when $\ta=\Sets_M$ is the category of right $M$-sets, for a monoid $M$. It is well-known (see for example Lemma 2.4.1 \cite{reconstruction}) that there is an order reversing bijection between the set of localising subcategories of ${\Loc}(\Sets_M)$ and all Grothendieck topologies defined on the one object category associated to $M$. Under this bijection, the localisation subcategory corresponding to a topology $\fa$, is the category of sheaves on $\fa$.

Let us recall the notion of a Grothendieck topology on a monoid and that of a sheaf over said Grothendieck topology. For this, we introduce the following notation: For a right ideal $\fa$ and an element $m\in M$, we set
$$(\fa:m)=\{x\in M| mx\in \fa\}.$$

For a monoid $M$, a \emph{Grothendieck topology} (or simply \emph{topology}) on a monoid $M$, is a collection of right ideals $\fa$, such that

\begin{itemize} 
	\item[(T1)]  $M\in \fa$,
	\item[(T2)]If $\af\in \fa$ and $m\in M$, then $(\af:m)\in \fa$,
	\item[(T3)] If $\fb\in \fa$ and $\af$ is a right ideal of $M$, such that $(\af:b)\in \fa$ for some $b\in \fb$, then $\af\in \fa$.
\end{itemize}

From these conditions we get the following results (see for example  \cite{bo_m}, \cite{reconstruction}):

\begin{itemize}
	\item[(i)] If $\af\subseteq \fb$ are right ideals and $\af\in \fa$, then $\fb\in \fa$.
	\item[(ii)] If $\af, \fb\in \fa$, then $\af\cap \fb\in \fa$.
	\item[(iii)] If $\af, \fb\in \fa$, then $\af\fb\in\fa$.
\end{itemize}

A right $M$-set $A$ is called an $\fa$-{\it sheaf} if the restriction map
$$A\to  Hom_M(\fa,A),\ \quad a\mapsto f_a$$
is a bijection for every $\fa\in \fa$. Here, $a\in A$ and $f_a\in Hom_M(\fa,A)$, where $f_a(x)=xa$, $x\in \fa$. We let $\Sh(\fa)$ denote the full subcategory of $\fa$-sheaves.

\begin{Le}\label{411} Let $\fa$ be a topology on a monoid $M$ and $A$ a filtered left $M$-set. Then
$${\sf p }_A\fork \Sh(\fa)$$
if and only if for any $\af\in \fa$, the canonical map $\af\otimes_MA\to A$ is an isomorphism. Here, ${\sf p}_A$ denotes the point of $\Sets_M$ corresponding to $A$.
\end{Le}

\begin{proof} Since the direct image functor $\Sets\to \Sets_M$ corresponding to the point ${\sf p}_A$ is given by $S\mapsto Hom_\Sets(A,S)$, we see that ${\sf p }_A\fork \Sh(\fa)$ if and only if $Hom_\Sets(A,S)$ is a $\fa$-sheaf. This means that, for all $\af\in \fa $, the canonical map
$$Hom_\Sets(A,S)\to Hom_{\Sets_M} (\af, Hom_\Sets(A,S))$$
is an isomorphism. The map in question is the same as 
$$Hom_\Sets(A,S)\to Hom_{\Sets} (\af\t _M A,S).$$
Since $S$ is any set, Yoneda's lemma implies that, this happens if and only if $\af\otimes_MA\to A$ is an isomorphism.
\end{proof}

\subsection{Idempotent ideals}

\begin{Le}\label{421.08} Let $\fm $ be a two-sided ideal of a monoid $M$, such that $\fm =\fm ^2$. The set
$$\cF_\fm =\{\fa \ | \ \fm\subseteq \fa\}$$
of right ideals containing $\fm $ is a Grothendieck topology on $M$.
\end{Le}

\begin{proof} The condition (T1) is obvious. For (T2), assume $ \fm  \subseteq \fa$ and $m \in M$. For any $x\in \fm $, $mx\in M\fm =\fm \subseteq \fa.$ Hence $x\in ( \fa:m )$. It follows that $\fm \subseteq (\fa:m)$ and $( \fa:m )\in \cF_\fm $. So (T2) holds. For (T3), take
$\fb\in \cF_\fm $. Assume $\af$ is a right ideal, such that $(\af:b)\in \cF_\fm $, for any $b\in \fb$. By assumption, $\fm \subseteq \fb$ and $\fm \subseteq (\af:b)$ for all $b\in \fb$. Take any $x\in \fm $. Since $\fm =\fm ^2$, we can write $x=yz$, where $y,z\in \fm $. Since $z\in \fm \subseteq (\af:y)$, we see that 
$x=yz\in \af$. Thus, $\fm \subseteq \af$ and $\af\in \fa_\fm $, from which (T3) follows.
\end{proof}

\begin{Pro}\label{422.08}
Let $\cF$ be a Grothendieck topology on a finite monoid $M$. There exists a two-sided ideal $\fm $, such that $\fm ^2=\fm $ and $\cF=\cF_\fm $.
\end{Pro}

\begin{proof} As $M$ is finite, $\cF $ is finite as well. We can also see that $\cF$ contains a smallest element $\fm$, since $\cF$ is closed with respect to finite intersection. Take $x\in M$. Since $(\fm :x)\in \cF$, it follows that $\fm \subseteq (\fm :x)$. Equivalently, $x\fm \subseteq \fm $, which implies that $\fm $ is a two-sided ideal of $M$. We also know that $\cF$ is closed with respect to the product. Hence $\fm ^2\in \cF$. By minimality of $\fm $, we have $\fm \subseteq \fm ^2$ and hence, $\fm ^2=\fm $.
\end{proof}

We will use the following, well-known, fact, see \cite[Proposition 1.23]{st}.

\begin{Le}\label{423.28} Any two-sided idempotent ideal of a finite monoide $M$ has the form
$$\bigcup_{i\in I}Me_iM,$$
for a (finite) family of idempotents $({e_i})_{i\in I}$.
\end{Le}

\begin{Le} Let $e$ be an idempotent of $M$ and $\fm \subseteq M$ a two-sided ideal, such that $\fm =\fm ^2$. Then
$${\sf p}_{Me}\fork \Sh(\cF_{\fm })$$
if and only if $e\in \fm $.
\end{Le}

\begin{proof} Assume ${\sf p}_{Me}\fork \Sh(\cF_{\fm })$. The canonical map
$$\mu_{\fm }:\fm \otimes_MMe\to Me$$
is an isomorphism by Lemma \ref{411}. The surjectivity of this map implies that $e=xme$, for some $x\in \fm $, $m\in M$. So, $e\in \fm  M=\fm $. Conversely, assume $e\in \fm $. We have to show that for any right ideal $\fa$, containing $\fm $, the canonical map $\mu_{\fa}:\fa\otimes_MMe\to Me$ is an isomorphism. For any $m\in M$, we have $me\in \fm \subseteq \fa$. Since $\mu(me\otimes e)=me$, we see taht $\mu$ is surjective. To show that $\mu$ is injective, we first observe that any element of $\fa\otimes_MMe$, can be write as $a\otimes e$, with $a\in \fa$. Assume $\mu(a\otimes e)=\mu(b\otimes e)$, $a,b\in \fa$. Thus $ae=be$. We have
$$a\otimes e=a\otimes ee=ae\otimes e=be\otimes e=b\otimes ee=b\otimes e$$
proving the injectivity of $\mu$.
\end{proof}

\begin{Exm}\label{425.08} i) Let $M$ once again be the monoid $\{1,0,a,b, ab\}$, $a^2=1=b^2$ and $ba=0$. It has 4 idempotents $1,0,a,b$. 
Since $M1=M$, $M0=\{0\}$, $Ma=\{0,a\}$ and $Mb=\{0,b,ab\}$, we see that they all define non-isomorphic idempotents. Thus, $M$ has $4$ non-isomorhic points corresponding to the filtered left $M$-sets $\{0\}, \{0,a\},\{0,b,ab\}$ and $M$. Since
$M1M=M$, $M0M=\{0\}$, $MaM=\{0,a,ab\}$, $MbM=\{0,b,ab\}$, we see that there are $6$ idempotent ideals
$$\emptyset, \{0\}, \{0,a,ab\},\{0,b,ab\}, \{0,a,b,ab\}, M.$$

ii) Let $M=T_3$ be the monoid of all endomorphic maps $\{1,2,3\}\to \{1,2,3\}$. It has $3^3=27$ elements. If $f$ is a such map,  the cardinality of the image of $f$ is known as the rank of $F$ and is denoted by $rk(f)$.
There are 3 maps of rank $1$ (i.e. constant maps)  and all of them are idempotents. There are $6$ idempotent elements of rank $2$ and only one idempotent of rank $3$. All together, we have $10$ idempotents. However, $\mathsf{Idem}_{\mathfrak{I}}(M)$ has only three elements, as two idempotents are $\mathfrak{I}$-equivalent if and only if they have the same rank. All two-sided ideals of $T_3$ are idempotent and they are 
$$\emptyset \subseteq I_1\subseteq I_2\subseteq I_3=M,$$
where $I_k=\{f\in T_3| rk(f)\leq k\}$, $k=1,2,3$. 
\end{Exm}

\subsection{Distributivity of the lattice ${\Loc}(\Sets_M)$}

Denote by $\ii(M)$ the set of  two-sided idempotent ideals of a monoid  $M$. Since the union of two-sided idempotent ideals of $M$ is again a two-sided idempotent ideal, we see that $\ii(M)$ is a join-semilattice, with $I\vee J=I\cup J$. Its greatest element is $M$ and least element is $\emptyset$. 
\\

In general, the intersection of two-sided idempotent ideals is not an idempotent ideal (see  i) of Example \ref{425.08}). This leads us to introduce the following notion:  We will say that monoid is \emph{III-closed} (idempotent ideal intersection), if for any two-sided idempotent ideals $I$ and $J$, the two-sided ideal $I\cap J$ is also an idempotent ideal.

\begin{Le}
\begin{itemize}
	\item[i)] Any finite regular monoid is III-closed.
	\item [ii)] Any finite commutative  monoid is III-closed.
\end{itemize}
\end{Le}

\begin{proof} i) This follows from the fact that any two-sided ideal in a finite regular monoid is an idempotent, see \cite[Corollary 1.25]{st}.  
	
ii) Let $I$ and $J$ be idempotent ideals of a finite commutative monoid $M$.  By Lemma \ref{423.28}, we can assume that $I$ is generated by $e_1,\cdots,e_m$ and $J$ is generated by $f_1,\cdots, f_n$, where $e_i,f_j$ are idemotents. Denote by $K$ the ideal generated by $e_if_j$. Clearly, $K\subseteq I\cap J$. Conversely, take $x\in  I\cap J$. We can write $x=ae_i=bf_j$, for some $a,b\in M$ and $1\leq i\leq m, 1\leq j\leq n$. We have
$$xf_j=ae_if_j=bf_j^2=bf_j=x.$$
As such, $x\in K$ and hence, $K=I\cap J$ is also generated by idempotents.
\end{proof}

Our interest in $\ii(M)$ comes from the bijection
\begin{equation}
	{\Loc}(\Sets_M)\cong \ii(M),
\end{equation}
which is true for all finite $M$. This follows from Lemmas \ref{421.08} and \ref{422.08}. In particular, $|{\Loc}(\Sets_M)|$ is finite and
\begin{equation}
|{\Loc}(\Sets_M)|=|{\Loc}(\Sets_{M^{op}})|
\end{equation} 
holds.
\\

It is well-known, that any finite join-semilattice $L$ with greatest element, is a lattice, where 
$$a\wedge b= \bigvee_ xx,$$
with $a,b\in L$ and $x$ running through all the elements of the set
$$\{x\in L| x\leq a \quad {\rm and} \quad x\leq b \}.$$
It follows that for finite $M$, the set $\ii(M)$ is a lattice.

\begin{Le}\label{424.08} Let $I$ and $J$ be two-sided idempotent ideals. Then
$$I\wedge J=\bigcup_{e} MeM,$$
where $e$ runs through all the idempotents of $I\cap J$.
\end{Le}

\begin{proof} The RHS is a two-sided ideal, generated by idempotents. Hence, it is an element of $\ii(M)$. If $e\in I\cap J$, it follows that $MeM\subseteq I$ and $MeM\subseteq J$. Thus, RHS $\subseteq$ LHS. Conversely, assume $K$ is a two-sided idempotent ideal, such that $K\subseteq I$ and $K\subseteq J$. By Lemma \ref{423.28}, we can write $K=\bigcup_{i\in I}Me_iM$ for some idempotents $e_i$. We have $e_i\in I\cap J$ by assumption, which implies that LHS is also a subset of RHS. The result follows.
\end{proof}

\begin{Co} If $M$ is finite, then $\ii(M)$ is a distributive lattice.
\end{Co}

\begin{proof} Take $I,J,K\in \ii(M).$ We need to show that
$$I\wedge (J\cup K)=(I\wedge J)\cup (I\wedge K)$$
By Lemma \ref{424.08}, we see that LHS is generated as a two-sided ideal by the idempotent elements of $ I\cap (J\cup K)$, while the RHS is generated (as a two-sided ideal) by the idempotent elements of $I\cap J$ and $I\cap K$. The reseult follows.  
\end{proof}

\subsection{Topology on $\F_M$ }

We start by recalling the well-known relationship between (finite) distributive lattices, posets and topologies.

Any poset $P$ has a natural topology, called the \emph{order topology}, where a subset $S\subseteq P$ is open if $y\in P$ and $x\leq y$ imply $x\in P$. Thus, $\ope(P)$ is a distributive lattice and it is finite if $P$ is finite. It is well-known, that any finite distributive lattice $L$ is of this form, for a uniquely defined $P$ (see for example \cite[p.106]{enum}). Specifically, for $P={\sf Irr}(L)$, the subset of irreducible elements of $L$ (an element $x\in L$ is irreducible if $x=y\vee z$ implies $x=y$ or $x=z$). 

The poset of our interest is $\F_M$, the iso-classes of the topos points of $M$, where $M$ is finite. According to Corollary \ref{372.08}, we have
$$\F_M\cong \mathsf{Idem}_{\mathfrak{I}}(M).$$
This allows to work with $\mathsf{Idem}_{\mathfrak{I}}(M)$ instead. This set has a canonical order
$$e\leq _{\mathfrak{I}} f \quad {\rm if} \quad MeM\subseteq mfM.$$
Thus, we have a canonical order topology on $\mathsf{Idem}_{\mathfrak{I}}(M)$ and as such, on $\F_M$. 

\begin{Pro} One has a bijection
$$\ii(M)\to {\sf Open}(\mathsf{Idem}_{\mathfrak{I}}(M)).$$
\end{Pro}

\begin{proof} It suffices to show that irreducible elements of $\ii(M)$ are exactly ideals of the form $MeM$, where $e$ is an idempotent. This follows form the facts that for any element $I\in\ii(M)$, one has $I=\bigcup _{e\in I }MeM$, where $e$ is an idempotent. If $MeM=J\cup K$, then $e\in K$, or $e\in J$. We get that $MeM=J$, or $MeM=K$.
\end{proof}

\begin{center}

\end{center}

\end{document}